\documentclass[a4paper,reqno]{amsart}

\usepackage[T1]{fontenc}
\usepackage{lmodern}
\usepackage{microtype}
\usepackage{amssymb,amsmath,amsthm,mathtools}
\usepackage{enumitem}
\usepackage{xcolor}
\usepackage[colorlinks,citecolor=blue,linkcolor=blue,urlcolor=blue]{hyperref}
\usepackage{aliascnt}
\usepackage[nameinlink,capitalise,noabbrev]{cleveref}

\allowdisplaybreaks
\numberwithin{equation}{section}
\newtheorem{theorem}{Theorem}[section]
\newaliascnt{proposition}{theorem}
\newtheorem{proposition}[proposition]{Proposition}
\aliascntresetthe{proposition}
\newaliascnt{lemma}{theorem}
\newtheorem{lemma}[lemma]{Lemma}
\aliascntresetthe{lemma}
\newaliascnt{corollary}{theorem}
\newtheorem{corollary}[corollary]{Corollary}
\aliascntresetthe{corollary}
\newaliascnt{remark}{theorem}
\newtheorem{remark}[remark]{Remark}
\aliascntresetthe{remark}
\newaliascnt{example}{theorem}
\newtheorem{example}[example]{Example}
\aliascntresetthe{example}
\newcommand{\Ext}{\operatorname{Ext}}
\newcommand{\Hom}{\operatorname{Hom}}
\newcommand{\End}{\operatorname{End}}
\newcommand{\rad}{\operatorname{rad}}
\newcommand{\soc}{\operatorname{soc}}
\newcommand{\add}{\operatorname{add}}
\newcommand{\modu}{\operatorname{mod}}
\newcommand{\perpA}{{}^{\perp}\!A}
\newcommand{\stableHom}{\underline{\Hom}}

\title[Self-extensions and radical cube zero]
{The Auslander--Reiten Conjecture for Algebras with Radical Cube Zero}
\author[X. Zhang]{Xiaojin Zhang}
\address{School of Mathematics and Statistics, Jiangsu Normal University,
Xuzhou 221116, Jiangsu, P. R. China}
\email{xjzhang@jsnu.edu.cn}
\author[P. Zhou]{Panyue Zhou}
\address{School of Mathematics and Statistics, Changsha University of Science and Technology,  Changsha 410114, Hunan, P. R. China}
\email{panyuezhou@163.com}
\makeatletter
\@namedef{subjclassname@2020}{\textup{2020} Mathematics Subject Classification}
\makeatother
\hypersetup{
 pdftitle={The Auslander--Reiten Conjecture for Algebras with Radical Cube Zero},
 pdfauthor={Xiaojin Zhang},
 pdfsubject={Self-extensions and homological conjectures},
 pdfkeywords={Auslander--Reiten conjecture, radical cube zero, separated algebra}
}

\begin{document}
\begin{abstract}
Let $A$ be a split finite-dimensional algebra over a field whose radical
$J$ satisfies $J^3=0$, and let $s$ be the number of isomorphism classes
of simple $A$-modules. We prove that a non-projective module $M$ with
$\Ext_A^i(M,A)=0$ for all $i>0$ has a non-zero self-extension in some
degree between $1$ and $3s+1$. In particular, $A$ satisfies the
Auslander--Reiten conjecture, which asserts that every self-orthogonal
generator is projective. As a consequence, every finite-dimensional
algebra over an algebraically closed field with radical cube zero satisfies the Auslander-Reiten conjecture.
\end{abstract}
\maketitle

\section{Introduction}

Auslander and Reiten \cite{AR75} proposed the Auslander--Reiten
conjecture. It asserts that a finitely generated module $M$ over an Artin
algebra $A$ is projective provided that
\begin{equation}\label{eq:ARC}
 \Ext_A^i(M,M\oplus A)=0\hspace{2mm}\text{for every }i>0.
\end{equation}
Here a module $M$ is \emph{self-orthogonal} if
$\Ext_A^i(M,M)=0$ for every $i>0$. A module $G$ is a \emph{generator}
if $A\in\add G$, where $\add G$ denotes the full subcategory consisting
of direct summands of finite direct sums of copies of $G$. Thus the
conjecture is equivalent to the assertion that every self-orthogonal
generator is projective. For a self-injective algebra, orthogonality to
$A$ is automatic. In this case the assertion is Tachikawa's second
conjecture. Chen and Xi \cite{ChenXi25} discuss its relation to the
Nakayama conjecture.

The conjecture is known for several classes of algebras with small
Loewy length. Xu \cite{Xu13,Xu15} proved it for local Artin algebras and
for special biserial algebras whose radicals have cube zero. Ringel and
Zhang \cite[Theorem~1.5]{RZ22} proved the stronger local statement that
every non-projective semi-Gorenstein-projective module over a short
local algebra, meaning a local algebra whose radical has cube zero, has
a non-zero self-extension in degree one. For symmetric
algebras, the corresponding result when the radical has cube zero goes
back to Hoshino \cite{Hoshino84}. Hoshino \cite{Hoshino89} obtained
related results. Broader accounts of known cases are given by Chen, Hu,
Qin and Wang \cite{CHQW23} and by Chen and Xi \cite{ChenXi25}.

Other homological conjectures are also known for algebras of Loewy
length three. Green and Zimmermann-Huisgen \cite{GreenZH91} proved the
finiteness of the finitistic dimension for Artinian rings whose radical
has cube zero. Dr\"axler and Happel \cite{DH92} proved the generalized
Nakayama conjecture under the assumptions $J^{2\ell+1}=0$ and
$A/J^\ell$ representation-finite. Taking $\ell=1$ includes algebras with
$J^3=0$. These results do not by themselves yield the Auslander--Reiten
conjecture for each such algebra. The equivalence between the latter
conjecture and the generalized Nakayama conjecture concerns their
validity over all Artin algebras and is not a pointwise equivalence for
an individual algebra. Chen, Hu, Qin and Wang
\cite[Introduction]{CHQW23} explain this distinction.

We give a uniform argument in Loewy length at most three, together
with an explicit bound on the first non-zero self-extension. Recall
that the Loewy length of $A$ is the least integer $m$ such that $J^m=0$.
Thus $J^3=0$ means that $A$ has Loewy length at most three. A
finite-dimensional $k$-algebra is \emph{split} if the endomorphism ring
of every simple module is $k$. Equivalently, its basic algebra has
semisimple quotient isomorphic to $k^s$ for some $s$. Here an algebra is
\emph{basic} if its regular module contains, up to isomorphism, each
indecomposable projective module exactly once as a direct summand.
Every finite-dimensional algebra over an algebraically closed field is
split. We write $A\text{-}\modu$ for the category of finitely generated
left $A$-modules. Put
\[
 \perpA=\{X\in A\text{-}\modu\mid
       \Ext_A^i(X,A)=0\text{ for all }i>0\}.
\]
Modules in $\perpA$ are also called semi-Gorenstein-projective modules.
Thus $\perpA$ consists precisely of the modules that are orthogonal to
the regular module $A$ in every positive extension degree.

\begin{theorem}\label{thm:main}
Let $A$ be a split finite-dimensional $k$-algebra, $J=\rad A$,
and let $s$ be the number of isomorphism classes of simple
$A$-modules. Assume that $J^3=0$. For $M\in\perpA$, the following
conditions are equivalent:
\begin{enumerate}[label=\textnormal{(\roman*)},leftmargin=2.2em]
 \item $M$ is projective.
 \vspace{2mm}
 \item $\Ext_A^q(M,M)=0$ for $1\le q\le 3s+1$.
\end{enumerate}
\end{theorem}

The bound $3s+1$ is the upper bound obtained from our argument and
is not claimed to be optimal. Ringel and Zhang \cite{RZ22} obtain the
sharper bound $1$ when $s=1$. Initial vanishing is essentially different
from eventual vanishing, as illustrated by the examples in
Section~\ref{sec:main-proof}.

The main theorem immediately yields the following form of the Auslander--Reiten conjecture.

\begin{corollary}\label{cor:ARC}
Every split finite-dimensional algebra with radical cube zero
satisfies the Auslander--Reiten conjecture. In particular, this holds
for every finite-dimensional algebra over an algebraically closed
field with radical cube zero.
\end{corollary}

\begin{proof}
Condition \eqref{eq:ARC} implies that $M\in\perpA$ and that
\cref{thm:main}\textnormal{(ii)} holds.
\end{proof}

The paper is organized as follows. In Section~2, we recall the
homological properties of syzygies and orthogonal modules. In Section~3,
we study extensions by passing to the separated algebra. In Section~4,
we prove the main theorem and give its consequences and examples.

All modules in this paper are finitely generated left modules. For
a projective cover $P_X\twoheadrightarrow X$, we write
$\Omega_A X=\ker(P_X\to X)$ and call it the first \emph{syzygy} of
$X$. Higher syzygies are defined inductively. We write
$\stableHom_A(X,Y)$ for the quotient of $\Hom_A(X,Y)$ by the subgroup
of morphisms that factor through projective modules. The notation
$U\mid V$ means that $U$ is isomorphic to a direct summand of $V$.
Auslander, Reiten and Smal\o{} \cite{ARS} give standard background on
projective covers, stable categories and separated algebras.

\section{Syzygies and finite orthogonal sequences}

In this section, we collect the properties of syzygies and orthogonal modules that will be used later.

The first two lemmas hold over any finite-dimensional algebra.

\begin{lemma}\label{lem:syzygy}
Let $X\in\perpA$. Then $\Omega_A^nX\in\perpA$ for all $n\ge0$, and
\[
 \Ext_A^q(\Omega_A^nX,\Omega_A^nX)
 \cong \Ext_A^q(X,X)\quad(q\ge1).
\]
Consequently, any specified vanishing range of positive
self-extensions is inherited by all syzygies and their direct
summands. If $X$ is non-projective, then $\Omega_A X$ has a
non-projective indecomposable direct summand.
\end{lemma}

\begin{proof}
Membership in $\perpA$ follows by dimension shifting. Write
$X_1=\Omega_A X$ and choose a projective cover sequence
\[
 0\longrightarrow X_1\longrightarrow P\longrightarrow X
 \longrightarrow0.
\]
For $q\ge1$, dimension shifting in the first variable gives
\[
 \Ext_A^q(X_1,X_1)\cong\Ext_A^{q+1}(X,X_1).
\]
Since $\Ext_A^{>0}(X,P)=0$, the long exact sequence in the second
variable identifies the group on the right with $\Ext_A^q(X,X)$.
Iteration proves the formula. Additivity of $\Ext$ proves the
assertion about direct summands.

If $\Omega_A X$ were projective, then
$\Ext_A^1(X,\Omega_A X)=0$, since $X\in\perpA$. The projective cover
sequence would split, making $X$ projective. Krull--Schmidt now gives
the last assertion.
\end{proof}

\begin{lemma}\label{lem:stable-ext}
For $X\in\perpA$, any module $Y$, and $n\ge1$, there is an
isomorphism
\[
 \Ext_A^n(X,Y)\cong\stableHom_A(\Omega_A^nX,Y).
\]
\end{lemma}

\begin{proof}
It suffices to treat $n=1$, since syzygies of $X$ belong to $\perpA$.
For a projective cover sequence
$0\to\Omega_A X\xrightarrow{u}P\to X\to0$, one has
\[
 \Ext_A^1(X,Y)\cong
 \Hom_A(\Omega_A X,Y)/\operatorname{Im}\Hom_A(P,Y).
\]
The image in the denominator consists of all maps factoring through
projectives. Indeed, if $\Omega_A X\xrightarrow{a}Q\xrightarrow{b}Y$
is such a factorization, with $Q$ projective, then
$\Ext_A^1(X,Q)=0$ allows $a$ to extend across $u$. Thus $ba$ lies in
the displayed image. The reverse inclusion is immediate.
\end{proof}

\begin{proposition}\label{prop:chain}
Let $L\ge1$ and let $M\in\perpA$ be non-projective. Assume that
$\Ext_A^q(M,M)=0$ for $1\le q\le L$. There exist indecomposable
non-projective modules $X_0,\ldots,X_{L-1}$ such that:
\begin{enumerate}[label=\textnormal{(\roman*)},leftmargin=2.2em]
 \item {$X_0\mid\Omega_A M$ and $X_{i+1}\mid\Omega_A X_i$.}
 \item {$X_i\in\perpA$ and $\Ext_A^q(X_i,X_i)=0$ for $1\le q\le L$.}
 \item if $j>i$, then
 \[
 \stableHom_A(X_j,X_i)=0=\Ext_A^1(X_j,X_i).
 \]
 \item the $X_i$ are pairwise non-isomorphic.
\end{enumerate}
If $J^3=0$, then $J^2X_i=0$ for every $i$.
\end{proposition}

\begin{proof}
The non-projective summands required in (i) exist by
\cref{lem:syzygy}. The same lemma and additivity of $\Ext$ give (ii).
Projective covers, and hence their syzygies, are additive up to
isomorphism. Therefore
\[
 X_j\mid\Omega_A^{j-i}X_i\quad(j>i).
\]
Using \cref{lem:stable-ext} and (ii), we obtain
\[
 \stableHom_A(\Omega_A^{j-i}X_i,X_i)
 \cong\Ext_A^{j-i}(X_i,X_i)=0.
\]
Also,
\[
 \Ext_A^1(\Omega_A^{j-i}X_i,X_i)
 \cong\Ext_A^{j-i+1}(X_i,X_i)=0,
\]
because $j-i+1\le L$. Passing to direct summands proves (iii).
If $X_j\cong X_i$ for $j>i$, then (iii) forces the identity of $X_i$
to factor through a projective. This would make $X_i$ projective,
proving (iv).

Finally, $\Omega_A Z\subseteq JP$ for a projective cover $P\to Z$.
Thus $J^2\Omega_A Z\subseteq J^3P=0$ when $J^3=0$. Each $X_i$ is a
summand of a positive syzygy, giving the final assertion.
\end{proof}

\section{Extensions over the separated algebra}

In this section, we study extension groups after passing from an algebra with radical square zero to its separated algebra.

We first record how orthogonality passes to a quotient algebra.

\begin{lemma}\label{lem:quotient}
Let $I$ be a two-sided ideal of $A$, put $B=A/I$, and let $X,Y$ be
$B$-modules. Regard a $B$-module as an $A$-module through the quotient map
$A\twoheadrightarrow B$. This restriction of scalars, often called
inflation, induces an injection
\[
 \Ext_B^1(X,Y)\hookrightarrow\Ext_A^1(X,Y).
\]
Moreover, $\stableHom_A(X,Y)=0$ implies
$\stableHom_B(X,Y)=0$.
\end{lemma}

\begin{proof}
A short exact sequence of $B$-modules that splits over $A$ also
splits over $B$, since an $A$-linear map between $B$-modules is
$B$-linear. This proves the first assertion.

For the second, factor an arbitrary map $f:X\to Y$ as
$X\to P\xrightarrow{b}Y$, where $P$ is $A$-projective. Since $IY=0$,
the map $b$ annihilates $IP$, so the factorization passes through
$P/IP$. This module is projective over $B$, being a direct summand
of a finite direct sum of copies of $B$.
\end{proof}

For the remainder of this section, $B$ is a basic split
finite-dimensional algebra with radical $\mathfrak r$ satisfying
$\mathfrak r^2=0$. For a $B$-module $X$, we write
$\mathfrak rX$ for its radical, $X/\mathfrak rX$ for its top, and
$\soc_BX$ for its socle. Choose primitive orthogonal idempotents
$e_1,\ldots,e_s$, and set $S=\bigoplus_{i=1}^s ke_i$. Then
$B=S\oplus\mathfrak r$ and $S\cong B/\mathfrak r\cong k^s$.
The \emph{separated algebra} of $B$ is
\[
H=\begin{pmatrix}S&0\\ \mathfrak r&S\end{pmatrix}.
\]
An $H$-module is a triple $(U,V,\mu)$, where $U,V$ are $S$-modules
and $\mu:\mathfrak r\otimes_SU\to V$ is $S$-linear. Its
separated quiver has the vertices
$1^+,\cdots,s^+,1^-,\cdots,s^-$, with all arrows directed from a plus
vertex to a minus vertex. The algebra $H$ is the path algebra of this
quiver. Hence $H$ is hereditary. We write $K_0(H)$ for the Grothendieck
group of finite-dimensional $H$-modules. It is free of rank $2s$, with
the classes of the simple $H$-modules as a basis.

For a $B$-module $X$, multiplication defines the surjection
$\mu_X:\mathfrak r\otimes_S(X/\mathfrak rX)\to\mathfrak rX$. Set
\[
 F(X)=(X/\mathfrak rX,\mathfrak rX,\mu_X).
\]
On morphisms, $F$ takes the induced map on the top and the restriction
to the radical. This is the classical functor to the separated algebra considered
by Reiten \cite{Reiten75} and by Auslander, Reiten and Smal\o{}
\cite{ARS}. We record the properties needed below.

\begin{lemma}\label{lem:F}
The functor $F:B\text{-}\modu\to H\text{-}\modu$ is full, with
\[
 \ker\bigl(\Hom_B(X,Y)\longrightarrow\Hom_H(FX,FY)\bigr)
 \cong\Hom_S(X/\mathfrak rX,\mathfrak rY).
\]
It preserves indecomposability and projectivity and reflects
isomorphisms. For an indecomposable $X$, the module $FX$ is
projective if and only if $X$ is projective.
\end{lemma}

\begin{proof}
Choose $S$-module splittings $X=T_X\oplus\mathfrak rX$ and
$Y=T_Y\oplus\mathfrak rY$, identifying $T_X,T_Y$ with the tops.
An $H$-map $FX\to FY$ is a pair $(\alpha,\beta)$ satisfying
\[
 \beta\mu_X=\mu_Y(1_{\mathfrak r}\otimes\alpha).
\]
The map $t+u\mapsto\alpha(t)+\beta(u)$ is $B$-linear and lifts the
pair. A map lies in the kernel of $F$ exactly when its image is
contained in $\mathfrak rY$. It then annihilates $\mathfrak rX$ and
factors uniquely through the indicated $S$-map. These kernels form
an ideal $\mathcal I$ with $\mathcal I^2=0$.

If $P_i=Be_i$, multiplication identifies
$\mathfrak r\otimes_S Se_i$ with $\mathfrak rP_i$. Thus $F(P_i)$ is
the indecomposable projective of $H$ at $i^+$, and $F$ preserves
projectives. Fullness gives
\[
 \End_H(FX)\cong\End_B(X)/\mathcal I(X,X).
\]
For indecomposable $X$ this is a local ring, so $FX$ is
indecomposable. If $FX\cong FY$, lift an isomorphism and its inverse
to $f:X\to Y$ and $g:Y\to X$. Then $gf-1_X$ and $fg-1_Y$ belong
to $\mathcal I$. Since $\mathcal I^2=0$, both $gf$ and $fg$ are
invertible, whence $f$ is an isomorphism.

Finally, a non-zero $X$ has non-zero top. If $X$ is indecomposable
and $FX$ is projective, then $FX$ must be isomorphic to some
$F(P_i)$. Indeed, the other indecomposable projective $H$-modules have
zero first component. Reflection of isomorphisms gives $X\cong P_i$.
\end{proof}

The next comparison does not require either module to be
indecomposable or non-simple. Its correction term specifies exactly
when the extension spaces have equal dimension.

\begin{proposition}\label{prop:comparison}
Let $X,Y$ be $B$-modules and fix a projective cover sequence
\[
 0\longrightarrow K_X\longrightarrow P_X\xrightarrow{p_X}X
 \longrightarrow0.
\]
There is a short exact sequence of $k$-vector spaces
\begin{equation}\label{eq:comparison}
 0\longrightarrow\Ext_H^1(FX,FY)
 \longrightarrow\Ext_B^1(X,Y)
 \longrightarrow\Hom_S(K_X,\soc_B Y/\mathfrak rY)
 \longrightarrow0.
\end{equation}
In particular,
\begin{equation}\label{eq:ext-implication}
 \Ext_B^1(X,Y)=0\quad\Longrightarrow\quad\Ext_H^1(FX,FY)=0.
\end{equation}
If $\soc_B Y=\mathfrak rY$, the first map in
\eqref{eq:comparison} is an isomorphism.
\end{proposition}

\begin{proof}
The kernel $K_X\subseteq\mathfrak rP_X$ is semisimple. Since $p_X$
induces an isomorphism on tops and a surjection on radicals, there
is an exact sequence of $H$-modules
\begin{equation}\label{eq:H-cover}
 0\longrightarrow(0,K_X,0)\longrightarrow F(P_X)
 \longrightarrow FX\longrightarrow0.
\end{equation}
Here $F(P_X)$ is projective by \cref{lem:F}. Notice that the kernel
in \eqref{eq:H-cover} is $(0,K_X,0)$, whereas $F(K_X)=(K_X,0,0)$.
In particular, \eqref{eq:H-cover} is not obtained by assuming that
$F$ is exact.

Applying $\Hom_B(-,Y)$ to the projective cover of $X$ yields
\begin{equation}\label{eq:B-coker}
 \Ext_B^1(X,Y)\cong
 \Hom_S(K_X,\soc_B Y)/\mathcal R,
\end{equation}
where $\mathcal R$ is the image of restriction from
$\Hom_B(P_X,Y)$. Every such restriction has image in
$\mathfrak rY$, because $K_X\subseteq\mathfrak rP_X$.
Applying $\Hom_H(-,FY)$ to \eqref{eq:H-cover} gives
\begin{equation}\label{eq:H-coker}
 \Ext_H^1(FX,FY)\cong
 \Hom_S(K_X,\mathfrak rY)/\mathcal R.
\end{equation}
The image is the same $\mathcal R$ in both formulas. Fullness lifts
each map $FP_X\to FY$ to a map $P_X\to Y$. Restricting the
radical component to $K_X$ gives the same map.

Thus \eqref{eq:H-coker} is a subspace of \eqref{eq:B-coker}.
Their quotient is $\Hom_S(K_X,\soc_B Y/\mathfrak rY)$, because $S$
is semisimple. This proves \eqref{eq:comparison} and its consequences.
\end{proof}

\begin{remark}\label{rem:dimension}
For an indecomposable non-simple $Y$, one has
$\soc_B Y=\mathfrak rY$. Indeed, a simple submodule of the socle
disjoint from $\mathfrak rY$ would split off. To see this, choose a
complement to its image in the semisimple top and take the inverse image
in $Y$.
Consequently, \cref{prop:comparison} recovers the equality of
extension dimensions in this case.

More generally, let $t_X,r_X$ be the multiplicity vectors of the top
and radical of $X$, let $c_Y$ be that of $\soc_B Y/\mathfrak rY$,
and put $E_{ji}=[\mathfrak rP_i:S_j]$. Since $K_X$ has multiplicity
vector $Et_X-r_X$, formula \eqref{eq:comparison} gives
\[
\dim_k\Ext_B^1(X,Y)-\dim_k\Ext_H^1(FX,FY)
 =(Et_X-r_X)^Tc_Y.
\]
\end{remark}

\begin{example}\label{ex:nonexact}
Let $B=k[\varepsilon]/(\varepsilon^2)$ and $T=B/(\varepsilon)$.
The separated algebra is the path algebra of $1^+\to1^-$. Here
$F(T)=(k,0,0)$, and
\[
 \Ext_B^1(T,T)\cong k,
 \quad \Ext_H^1(FT,FT)=0.
\]
The correction term in \eqref{eq:comparison} is $k$. Moreover, $F$
sends the inclusion $T\cong\varepsilon B\hookrightarrow B$ to the
zero map. This exhibits both the failure of exactness of $F$ and the
need for a correction term when the target has a simple summand.
\end{example}

We will also use the following elementary hereditary facts.

\begin{lemma}\label{lem:projective-maps}
Let $H$ be a finite-dimensional hereditary algebra. If $U$ is
indecomposable and non-projective, then $\Hom_H(U,Q)=0$ for every
projective $Q$.
\end{lemma}

\begin{proof}
The image of a map $U\to Q$ is a submodule of a projective, hence
projective. If the image is non-zero, the epimorphism onto it splits.
Indecomposability would then make $U$ projective.
\end{proof}

\begin{corollary}\label{cor:hom-transfer}
Let $X$ be an indecomposable non-projective $B$-module. Then
\[
 \stableHom_B(X,Y)=0\quad\Longrightarrow\quad\Hom_H(FX,FY)=0.
\]
\end{corollary}

\begin{proof}
Lift an $H$-map by fullness of $F$. The resulting $B$-map factors
through a $B$-projective, so its image under $F$ factors through an
$H$-projective. Since $FX$ is indecomposable and non-projective,
\cref{lem:projective-maps} makes this factorization zero.
\end{proof}

\begin{lemma}\label{lem:rank}
Let $H$ be a finite-dimensional hereditary algebra with $N$ simple
modules up to isomorphism. Suppose that non-zero modules
$Y_0,\ldots,Y_{r-1}$ satisfy
\[
 \Ext_H^1(Y_j,Y_i)=0~(j\ge i),
 \quad\Hom_H(Y_j,Y_i)=0~(j>i).
\]
Then $r\le N$.
\end{lemma}

\begin{proof}
The Euler form on $K_0(H)$ is the bilinear form
\[
 \langle[U],[V]\rangle_H=
 \dim_k\Hom_H(U,V)-\dim_k\Ext_H^1(U,V).
\]
It is well-defined by the long exact sequences and heredity of $H$.
Put $d_i=[Y_i]$. The matrix
\[
 G=(\langle d_j,d_i\rangle_H)_{0\le j,i<r}
\]
is upper triangular and has diagonal entries
$\dim_k\End_H(Y_i)>0$. Thus $G$ has rank $r$ over $\mathbb Q$.
On the other hand, if $D$ is the matrix whose columns are the $d_i$
in a basis of $K_0(H)$ and $C$ is the Euler matrix in that basis,
then $G=D^TCD$. Hence $\operatorname{rank}G\le N$.
\end{proof}

\section{Proof and consequences}\label{sec:main-proof}

In this section, we prove the main theorem and then give several consequences and examples.

\begin{proof}[\bf \emph{Proof of \cref{thm:main}}]
Projectivity implies the required vanishing. For the converse, the
assertion is invariant under Morita equivalence. Indeed, an exact
equivalence preserves extension groups, projectives and the number
$s$ of simple modules. It sends the regular module to a projective
generator, whose additive closure is the full subcategory of
projectives, so it preserves the condition defining $\perpA$.
It also preserves radical filtrations. Hence $J^3=0$ is preserved.
We may therefore assume that $A$ is basic.

Suppose that $M\in\perpA$ is non-projective and
$\Ext_A^q(M,M)=0$ for $1\le q\le L$, where $L=3s+1$.
Take $X_0,\ldots,X_{L-1}$ from \cref{prop:chain}, and put
\[
 B=A/J^2.
\]
The modules $X_i$ are annihilated by $J^2$, so they are $B$-modules.
They remain indecomposable and pairwise non-isomorphic, because
the morphisms between $B$-modules are unchanged by inflation to $A$.
By \cref{lem:quotient},
\begin{equation}\label{eq:B-orthogonal}
 \Ext_B^1(X_j,X_i)=0\quad(j\ge i),
 \quad\stableHom_B(X_j,X_i)=0\quad(j>i).
\end{equation}

There are exactly $s$ isomorphism classes of indecomposable
projective $B$-modules. Since the $X_i$ are pairwise non-isomorphic,
at most $s$ of them are projective over $B$. Retain all remaining
terms, with indices $i_0<\cdots<i_{r-1}$, and set
\[
 Y_a=F(X_{i_a})\quad(0\le a<r).
\]
Then $r\ge L-s=2s+1$. By \cref{lem:F}, each $Y_a$ is
indecomposable and non-projective. Formulas
\eqref{eq:ext-implication} and \eqref{eq:B-orthogonal} give
\[
 \Ext_H^1(Y_b,Y_a)=0\quad(b\ge a),
\]
including when the original $X_{i_a}$ is simple. Similarly,
\cref{cor:hom-transfer} gives
\[
 \Hom_H(Y_b,Y_a)=0\quad(b>a).
\]
The separated algebra $H$ has $2s$ simple modules. Therefore
\cref{lem:rank} gives $r\le2s$, contradicting $r\ge2s+1$.
\end{proof}

\begin{corollary}\label{cor:generator}
Let $A$ be a split finite-dimensional algebra with $J^3=0$.
Every self-orthogonal generator is projective.
\end{corollary}

\begin{proof}
If $A\in\add G$ and $\Ext_A^{>0}(G,G)=0$, then
$\Ext_A^{>0}(G,A)=0$. Apply \cref{cor:ARC}.
\end{proof}

\begin{corollary}\label{cor:TC2}
Let $A$ be a split finite-dimensional self-injective algebra with
$J^3=0$, and let $s$ denote its number of simple modules up to
isomorphism. If $\Ext_A^q(M,M)=0$ for $1\le q\le3s+1$, then $M$ is
projective. In particular, Tachikawa's second conjecture holds for
$A$.
\end{corollary}

\begin{proof}
Self-injectivity gives $\Ext_A^{>0}(M,A)=0$ for every $M$.
Apply \cref{thm:main}.
\end{proof}

The proof also gives a criterion for individual modules over
algebras of larger Loewy length.

\begin{proposition}\label{prop:eventual-syzygies}
Let $A$ be a split finite-dimensional algebra with $s$ simple modules
up to isomorphism, without a restriction on $J^3$. Let $M\in\perpA$.
Suppose that there is an integer $n_0\ge0$ such that
\[
 J^2\Omega_A^nM=0\hspace{2mm}\text{for all }n\ge n_0.
\]
If $\Ext_A^q(M,M)=0$ for $1\le q\le3s+1$, then $M$ is projective.
\end{proposition}

\begin{proof}
If $M$ were non-projective, repeated use of \cref{lem:syzygy} would
make $\Omega_A^{n_0}M$ non-projective and preserve the given
vanishing range. Apply \cref{prop:chain} to this module. Every term
is a summand of some $\Omega_A^{n_0+i+1}M$ and is therefore
annihilated by $J^2$. The rest of the proof of \cref{thm:main}, with
the same quotient $A/J^2$, applies verbatim.
\end{proof}

\begin{remark}\label{rem:scope}
The assumption $M\in\perpA$ in \cref{thm:main} still requires
orthogonality to $A$ in every positive degree. Only the
self-extension condition has been replaced by a finite test.
In \cref{prop:eventual-syzygies}, all sufficiently high syzygies
must be annihilated by $J^2$. The proof does not use annihilation of a
single syzygy as a substitute for this hypothesis.
\end{remark}

\begin{example}\label{ex:cycle}
Let $Q_s$ be the oriented cycle on $s\ge2$ vertices, let $R$ be its
arrow ideal, and put $A_s=kQ_s/R^2$. Label its simple left modules
$S_i$ with indices modulo $s$ so that $\rad P_i\cong S_{i+1}$.
This is a self-injective Nakayama algebra, and
\[
 \Omega_{A_s}^nS_i\cong S_{i+n}.
\]
The simples are non-projective. Since the identity of a
non-projective simple does not factor through a projective,
\cref{lem:stable-ext} gives
\[
 \dim_k\Ext_{A_s}^n(S_i,S_i)=
 \begin{cases}1,&s\mid n,\\0,&s\nmid n\end{cases}
 \quad(n\ge1).
\]
Thus a uniform bound valid for algebras with $s$ simple modules
cannot be smaller than $s$. In particular, the bound $1$ from the
local case does not extend to the non-local case.

An example of Loewy length exactly three is $kQ_3/R^3$. Its
projectives have successive composition factors
$S_i,S_{i+1},S_{i+2}$, and it is again self-injective. One has
$\Omega^2S_i\cong S_i$, so
\[
 \Ext^1(S_i,S_i)=0,
 \quad\Ext^2(S_i,S_i)\cong k.
\]
\end{example}

\begin{example}\label{ex:ext-finite}
Let $q\in k^\times$ have infinite multiplicative order and set
\[
 \Lambda=k\langle x,y\rangle/(x^2,y^2,xy+qyx).
\]
This four-dimensional self-injective algebra has radical cube zero.
For $\lambda\in k^\times$, let $C(\lambda)$ have basis $u,v$ with
$xv=u$, $yv=\lambda u$, and $xu=yu=0$. A standard calculation due to Erdmann
\cite[Section~4.1]{Erdmann17} gives
\[
 \Omega C(\mu)\cong C(q^{-1}\mu),\qquad
 \dim_k\Hom_\Lambda(C(\mu),C(\lambda))=1+\delta_{\mu,\lambda},
\]
where $\delta$ is the Kronecker delta. Applying $\Hom(-,C(\lambda))$
to $$0\to C(q^{-1}\mu)\to\Lambda\to C(\mu)\to0$$ gives
\[
 \dim_k\Ext_\Lambda^1(C(\mu),C(\lambda))
 =\delta_{\mu,\lambda}+\delta_{\mu,q\lambda}.
\]
Dimension shifting therefore yields
\[
 \dim_k\Ext_\Lambda^n(C(\lambda),C(\lambda))
 =\begin{cases}1,&n=1,\\0,&n\ge2.\end{cases}
\]
The module $C(\lambda)$ is non-projective. Hence eventual vanishing
of self-extensions cannot replace the initial vanishing hypothesis
of \cref{thm:main}, even for a local self-injective algebra.
\end{example}

Our argument leaves a gap between the necessary bound $s$ in
\cref{ex:cycle} and the sufficient bound $3s+1$ in \cref{thm:main}.
Improving this estimate and finding further classes satisfying the
hypothesis of \cref{prop:eventual-syzygies} are questions left by
the proof.

\textbf{Acknowledgements.} Xiaojin Zhang is supported by the National Natural Science Foundation of China (Grant Nos. 12171207 and 12371038). Panyue Zhou is supported by the National Natural Science Foundation of China (Grant No. ~12371034).

\end{document}